\documentclass{AIMS}
\usepackage{amsmath}
  \usepackage{paralist}
 \usepackage[colorlinks=true]{hyperref}
\hypersetup{urlcolor=blue, citecolor=red}

\usepackage{amssymb}
\usepackage[all]{xy}

\def\al{\alpha} 
\def\be{\beta} 
\def\ga{\gamma}

\def\th{\theta} 
 
\def\ka{\kappa} 
\def\la{\lambda}

\def\ta{\tau} 
\def\ph{\varphi} 
 
\def\ps{\psi} 
 
\def\Ga{\Gamma}

\def\Ph{\Phi} 
\def\Ps{\Psi} 
\def\Om{\Omega}
 
\def\o{\circ} 
\def\inv{^{-1}} 
\def\x{\times}
\def\p{\partial}

\def\R{{\mathbb R}}

\def\exp{\operatorname{exp}}

\let\on=\operatorname

\let\mb=\mathbb
\let\mc=\mathcal

\newcommand{\ud}{\,\mathrm{d}}

\newtheorem{theorem}[subsection]{Theorem}
\newtheorem{corollary}[subsection]{Corollary}

\newtheorem{lemma}[subsection]{Lemma}
\newtheorem{proposition}[subsection]{Proposition}

\theoremstyle{definition}
\newtheorem{definition}[subsection]{Definition}
\newtheorem{remark}[subsection]{Remark}

\newtheorem*{openquestion*}{Open Question}

\newcommand{\ep}{\varepsilon}

\title[Completeness for Sobolev metrics] 
      {Completeness Properties of Sobolev Metrics on the Space of Curves}

\author[Martins Bruveris]{}


 \email{martins.bruveris@brunel.ac.uk}

\begin{document}
\maketitle

\centerline{\scshape Martins Bruveris }
\medskip
{\footnotesize
 \centerline{Department of Mathematics}
   \centerline{Brunel Unversity London}
   \centerline{Uxbridge UB8 3PH, United Kingdom}
} 

\bigskip

 \centerline{(Communicated by the associate editor name)}

\begin{abstract}
We study completeness properties of Sobolev metrics on the space of immersed curves and on the shape space of unparametrized curves. We show that Sobolev metrics of order $n\geq 2$ are metrically complete on the space $\mc I^n(S^1,\R^d)$ of Sobolev immersions of the same regularity and that any two curves in the same connected component can be joined by a minimizing geodesic. These results then imply that the shape space of unparametrized curves has the structure of a complete length space.
\end{abstract}

\section{Introduction}

The purpose of this paper is to continue the study of completeness properties of Sobolev metrics on the space of closed curves, which was initiated in \cite{Bruveris2014}.
 Sobolev metrics on spaces of curves were introduced independently in \cite{Charpiat2007,Mennucci2007,Michor2006c} and applied to problems in computer vision and shape analysis. They were generalized to immersed higher-dimensional manifolds in \cite{Bauer2011b}. See \cite{Bauer2014} for an overview of their properties and how they relate to other metrics used in shape analysis.

The arguably simplest Riemannian metric on the space $\on{Imm}(S^1,\R^d)$ of smooth closed curves is the $L^2$ metric given by
\[
G_c(h_1,h_2) = \int_{S^1} \langle h_1, h_2 \rangle \ud s\,;
\]
here $c$ is a curve and $h_1,h_2 \in T_c\on{Imm}(S^1,\R^d)$ are tangent vectors. We integrate with respect to arc length, $\ud s = |c'|\ud \th$, in order for the metric to be invariant under the reparametrisation action $(\ph, c) \mapsto c \o \ph$. It was shown in \cite{Michor2005,Bauer2012c} that the geodesic distance induced by the $L^2$-metric vanishes identically, rendering it unsuitable for applications. The quest for stronger metrics has led to the class of almost-local metrics \cite{Shah2008, Michor2006c} as well as the class of Sobolev metrics, which are the object of this work. Sobolev metrics are metrics of the form
\[
G_c(h_1,h_2) = \int_{S^1} a_0 \langle h_1, h_2 \rangle +
a_1 \langle D_s h_1, D_s h_2 \rangle + 
\dots + a_n \langle D_s^n h_1, D_s^n h_2 \rangle \ud s\,,
\]
with $a_j \geq 0$ and $D_s h = h'/|c'|$ denoting differentiation with respect to arc length. For the purposes of this article we will assume that the coefficients $a_j$ are constant.

Sobolev metrics of order $n$ possess various nice properties: the geodesic equation is locally ($n \geq 1$) and globally well-posed ($n \geq 2$); the geodesic distance is nonvanishing ($n \geq 1$) and for some particular metrics geodesics can be computed explicitely. Of particular interest for applications are first order metrics, because they permit geodesics to be computed effectively. The geodesic equation of a Sobolev metric of order $n$ is given by
\begin{multline*}
\p_t \left(\sum_{j=0}^n (-1)^j a_j \,|c'|\, D_s^{2j} c_t\right) = \\
= -\frac{a_0}2 \,|c'|\, D_s\left( \langle c_t, c_t \rangle v \right)
+ \sum_{k=1}^n \sum_{j=1}^{2k-1} (-1)^{k+j} \frac{a_k}{2}\, |c'|\, D_s
\left(\langle D_s^{2k-j} c_t, D_s^j c_t \rangle v \right)\,,
\end{multline*}
and one can see that it is a nonlinear PDE of order $2n$; see \cite{Bruveris2014,Michor2007} for a derivation. First order metrics without an $L^2$-term admit a remarkable transformation that maps immersions modulo translations isometrically to a codimension 2 submanifold of a flat space. This transformations was exploited in \cite{Michor2008a,Bauer2014c,Srivastava2011} to construct efficient numerical methods for computing geodesic distances between curves. Some attempts have been made in \cite{Bauer2014b} to generalize these transformations to higher order Sobolev metrics.

A drawback of first order metrics is that they are not complete. Geodesics can cease to exist after finite time and numerical computations show that geodesics need not exist between two curves. This motivates the study of higher order metrics as was done in \cite{Michor2006c, Mennucci2008,Bruveris2014}.

In particular we focus our attention on completeness properties of Sobolev metrics of order two and higher. For a Riemannian manifold $(M,g)$ there are three notions of completeness.
\begin{enumerate}[(A)]
\item
$(M,\on{dist})$ with the geodesic distance is a complete metric space;
\item
All geodesics can be extended for all time;
\item
Any two points can be joined by a minimizing geodesic.
\end{enumerate}
Property (A) is called metric completeness and (B) is geodesic completeness. In finite dimensions the theorem of Hopf--Rinow asserts that metric and geodesic completeness are equivalent and that either of them implies (C). In infinite dimensions for strong Riemannian manifolds\footnote{An infinite-dimensional Riemannian manifold $(M,g)$ is called \emph{strong}, if $g$ induces the natural topology on each tangent space or equivalently, if the map $g: TM \to (TM)'$ is an isomorphism. If $g$ is merely a smoothly varying nondegenerate bilinear form on $TM$ we call $(M,g)$ a \emph{weak} Riemannian manifold, indicating that the topology induced by $g$ can be weaker than the natural topology on $TM$ or equivalently $g:TM \to (TM)'$ is only injective.} one only has that metric completeness implies geodesic completeness.

It was shown in \cite{Bruveris2014} that $\on{Imm}(S^1,\R^2)$ and $\mc I^n(S^1,\R^2)$, the space of Sobolev immersions of order $n$, are geodesically complete for a Sobolev metric with constant coefficients of order $n\geq 2$. In \cite{Bauer2014_preprint} it is remarked that the same method also implies metric completion of $\mc I^n(S^1,\R^2)$ and \cite{Vialard2014_preprint} shows the existence of minimizing geodesics in $\mc I^n(S^1,\R^2)$. Similar results weere obtained in \cite{Bruveris2014_preprint} for diffeomorphism groups of $\R^d$ and compact manifolds.

We extend the completeness results from plane curves to curves in space and provide a different proof for the existence of minimizing geodesics. We also study the completeness of the quotient space of unparametrized curves.

\subsection{Contributions}

This paper provides a discussion of completeness properties of the spaces of parametrized and unparametrized curves in $\R^d$, equipped with Sobolev metrics. In Sect.~\ref{sec:estimates} we show the main estimate for Sobolev metrics of order $n \geq 2$ with constant coefficients. If $G$ is such a metric on the space $\mc I^n(S^1,\R^d)$ of Sobolev immersions and $B(c_0,r)$ is a metric ball with respect to the induced geodesic distance, then there exists a constant $C = C(c_0,r)$, such that
\[
C\inv \| h \|_{H^n(d\th)} \leq \sqrt{G_c(h,h)} \leq C \| h \|_{H^n(d\th)}
\]
holds for all $c \in B(c_0,r)$. Here $\|\cdot\|_{H^n(d\th)}$ is the inner product defining the topology of $\mc I^n(S^1,\R^d)$. In other words, the inner product defined by $G$ is equivalent to the ambient inner product with a constant that is uniform on metric balls. This is the content of Prop.~\ref{prop:sobolev_uniform_equivalent}, which is a generalization of \cite[Lem. 5.1]{Bruveris2014} from plane curves to curves in $\R^d$. Equivalence is clear for strong Riemannian metrics, the important part is the uniformity of the constant.

The uniform equivalence is used in Sect.~\ref{sec:met_geod_complete} to show that the inequality
\[
\| c_1 - c_2 \|_{H^n(d\th)} \leq C \on{dist}(c_1,c_2)
\]
holds on metric balls with respect to the geodesic distance. Thus, on metric balls, the natural vector space distance on $H^n(S^1,\R^d)$ is Lipschitz with respect to the geodesic distance. This allows us to show that $\mc I^n(S^1,\R^d)$ is metrically and hence geodesically complete, thus extending the result of \cite{Bruveris2014} on geodesic completeness from plane curves to curves in $\R^d$. With an approximation argument we then show in Thm.~\ref{thm:immersion_weak_completion} that the metric completion of the space $\on{Imm}(S^1,\R^d)$ of smooth immersions is equal to $\mc I^n(S^1,\R^d)$. However, since a geodesic with smooth initial conditions remains smooth, the space $\on{Imm}(S^1,\R^d)$ is geodesically complete. This provides a family of geodesically, but not metrically complete (weak) Riemannian manifolds.

In Sect.~\ref{sec:min_geodesics} we show that any two curves in the same connected component  can be connected by a minimizing geodesic. The proof exploits the structure of the arc length differentiation operator $D_s$ to prove a statement about its continuity under weak convergence. The method of proof is different from \cite{Vialard2014_preprint}, which relied instead on reparametrizing curves to constant speed. We also discuss possible extensions of the proof to Sobolev metrics with non-constant coefficients. The question whether the minimizing geodesic joining smooth curves is itself smooth remains open.

We transfer in Sect.~\ref{sec:shape_space} the results from the space of parametrized curves to the shape space of unparametrized curves. Denote by
\[
\mc B^n(S^1,\R^d) = \mc I^n(S^1,\R^d)/ \mc D^n(S^1)\,,
\]
the shape space of unparametrized Sobolev curves. Then $\mc B^n(S^1,\R^d)$ is not a manifold any more, but, equipped with the projection of the geodesic distance, it is a complete metric space. It is also the metric completion of the shape space of smooth immersions,
\[
B_i(S^1,\R^d) = \on{Imm}(S^1,\R^d) / \on{Diff}(S^1)\,.
\]
The distance in $\mc B^n(S^1,\R^d)$ is always realized by geodesics in $\mc I^n(S^1,\R^d)$ in the following sense: given $c_1,c_2 \in \mc I^n(S^1,\R^d)$, there exists $\ps \in \mc D^n(S^1)$, such that
\[
\on{dist}_{\mc B}(\pi(c_1), \pi(c_2)) = \inf_{\ph \in \mc D^n(S^1)} \on{dist}_{\mc I}(c_1, c_2 \o \ph) = \on{dist}_{\mc I}(c_1, c_2 \o \ps)
\]
and $c_1$ and $c_2 \o \ps$ can be joined by a minimizing geodesic. Furthermore $(\mc B^n, \on{dist}_{\mc B})$ carries the structure of a complete length space.

\section{Background Material and Notation}

\subsection{The Space of Curves}

Let $d \geq 1$. The space
\[
\on{Imm}(S^1, \R^d) = \left\{ c \in C^\infty(S^1, \R^d) \,:\, c'(\th) \neq 0 \right\}
\]
of immersions or regular, parametrized curves is an open set in the Fr\'echet space $C^\infty(S^1, \R^d)$ with respect to the $C^\infty$-topology and thus itself a smooth Fr\'echet manifold. For $s \in \R$ and $s > 3/2$ the space
\[
\mc I^s(S^1,\R^d) = \left\{ c \in H^s(S^1,\R^d) \,:\, c'(\th) \neq 0 \right\}
\]
of Sobolev curves of order $s$ is similarly an open subset of $H^s(S^1,\R^d)$ and hence a Hilbert manifold. Because of the Sobolev embedding theorem \cite{Adams2003}, $\mc I^s(S^1,\R^d)$ is well-defined and each curve in $\mc I^s(S^1,\R^d)$ is a $C^1$-immersion. To simplify notation we will sometimes omit the domain and image of the function spaces and write $\on{Imm}$ and $\mc I^s$ for the spaces $\on{Imm}(S^1,\R^d)$ and $\mc I^s(S^1,\R^d)$ respectively.

As open subsets of vector spaces the tangent bundles of the spaces $\on{Imm}(S^1,\R^d)$ and $\mc I^s(S^1,\R^d)$ are trivial,
\begin{align*}
T\on{Imm}(S^1,\R^d) &\cong \on{Imm}(S^1,\R^d) \x C^\infty(S^1,\R^d) \\
T\mc I^s(S^1,\R^d) &\cong \mc I^s(S^1,\R^d) \x H^s(S^1,\R^d)\,.
\end{align*}
From a geometric perspective the tangent space at a curve $c$ consists of vector fields along it, i.e., $T_c \on{Imm} = \Ga(c^\ast T\R^d)$. In the Sobolev case, where $c \in \mc I^s$, the pullback bundle $c^\ast T\R^d$ is not a $C^\infty$-manifold and the tangent space consists of fibre-preserving $H^s$-maps,
\begin{equation*}
T_c \mc I^s(S^1,\R^2) = 
\left\{h \in H^s(S^1,T\R^d): \quad \begin{aligned}\xymatrix{
& T\R^d \ar[d]^{\pi} \\
S^1 \ar[r]^c \ar[ur]^h & \R^d
} \end{aligned} \right\}\,.
\end{equation*}
See \cite{Michor1997,Hamilton1982} for details in the smooth case and \cite{Eells1966, Palais1968} for spaces of Sobolev maps.

For a curve $c \in \mc I^s(S^1,\R^d)$ or $c \in \on{Imm}(S^1,\R^d)$ we denote the parameter by $\th \in S^1$ and differentiation $\p_\th$ by $'$, i.e., $h' = \p_\th h$. Since $c$ is a $C^1$-immersion, the unit-length tangent vector $v = c'/|c'|$ is well-defined. We will denote by $D_s = \p_\th / |c'|$ the derivative with respect to arc length and by $\ud s = |c'| \ud \th$ the integration with respect to arc length. To summarize, we have
\begin{align*}
v &= D_s c\,, &
D_s & = \frac{1}{|c'|} \p_\th\,, &
\ud s & = |c'| \ud \th\,.
\end{align*}
We will write $D_c$ for $D_s$ in Sect.~\ref{sec:min_geodesics} to emphasize the dependence of the arc length derivative on the underlying curve. The length of $c$ is denoted by $\ell_c = \int_{S^1} 1 \ud s$.

\subsection{Sobolev Norms}

In this paper we will only consider Sobolev metrics of integer order. Sometimes it will be necessary to work with Sobolev spaces of fractional order and some of the results, which involve only the topology, are true also for fractional orders. We will denote by $n \in \mb N$ the order of the metric and we will use $s \in \R$, whenever fractional Sobolev orders are allowed or needed.

For $n \geq 1$ we fix the following norm on $H^n(S^1,\R^d)$,
\begin{equation*}
\| h \|^2_{H^n_\th} = \| h \|_{H^n(d\th)}^2 = \int_{S^1} |h(\th)|^2 + |\p_\th^n h(\th)|^2 \ud \th\,.
\end{equation*}
Its counterpart is the $H^n(ds)$-norm
\begin{equation*}
\| h \|_{H^n(ds)}^2 = \int_{S^1} |h(s)|^2 + |D_s^{n}h(s)|^2 \ud s\,,
\end{equation*}
which depends on the curve $c \in \mc I^n(S^1,\R^d)$. The norms $H^n(d\th)$ and $H^n(ds)$ are equivalent, but the constant in the inequalities
\[
C\inv \| h \|_{H^n(d\th)} \leq  \| h \|_{H^n(ds)} \leq C  \| h \|_{H^n(d\th)}
\]
depends on $c$. We will show in Prop.~\ref{prop:sobolev_uniform_equivalent} that if $c$ remains within a certain bounded set, then the constant can be chosen independently of the curve.

The $L^2(d\th)$- and $L^2(ds)$-norms are defined similarly,
\begin{align*}
\| u \|^2_{L^2(d\th)} &= \int_{S^1} |u|^2 \ud \th\,, &
\| u \|^2_{L^2(ds)} = \int_{S^1} |u|^2 \ud s\,,
\end{align*}
and they are related via $\left\| u \sqrt{|c'|} \right\|_{L^2(d\th)} = \| u \|_{L^2(ds)}$.

\subsection{Poincar\'e Inequalities}

The first part of the following lemma is a Sobolev embedding theorem with explicit constants and can be found in \cite{Mennucci2008}. The importance of the last part is that it contains no constant depending on $c$, even though it is a statement about arc length derivatives and the $L^2(ds)$-norms. The proofs can be found in \cite[Lem. 2.14]{Bruveris2014} and \cite[Lem. 2.15]{Bruveris2014}.

\begin{lemma} 
\label{lem:poincare_inequalities}
Let $c \in \mc I^2(S^1,\R^d)$ and $h \in H^1(S^1,\R)$. Then
\[
\| h\|_{L^\infty}^2 \leq \displaystyle\frac 2{\ell_c} \| h \|_{L^2(ds)}^2 + \displaystyle \frac {\ell_c} 2  \| D_s h \|_{L^2(ds)}^2\,,
\]
and if $h \in H^2(S^1,\R)$, then
\[
\| D_s h \|^2_{L^\infty} \leq \frac{\ell_c} 4 \| D_s^2 h \|_{L^2(ds)}^2\,.
\]
If $n \geq 2$, $c \in \mc I^n(S^1,\R^d)$ and $h \in H^n(S^1,\R)$, then for $0 \leq k \leq n$,
\[
\| D_s^k h \|^2_{L^2(ds)} \leq \| h\|^2_{L^2(ds)} + \| D_s^n h \|^2_{L^2(ds)}\,.
\]
\end{lemma}

\subsection{Gronwall Inequalities}

The following version of Gronwall's inequality can be found in \cite[Thm.~1.3.2]{Pachpatte1998} and \cite{Jones1964}.

\begin{theorem}
\label{thm:gronwall}
Let $A$, $\Ph$, $\Ps$ be real continuous functions defined on $[a,b]$ and $\Ph \geq 0$. We suppose that on $[a,b]$ we have the following inequality
\[
A(t) \leq \Ps(t) + \int_a^t A(s)\Ph(s) \ud s\,.
\]
Then
\[
A(t) \leq \Ps(t) + \int_a^t \Ps(s)\Ph(s) \exp\left(\int_s^t \Ph(u) \ud u \right) \ud s
\]
holds on $[a,b]$.
\end{theorem}

We will make use of the following corollary.

\begin{corollary}
\label{cor:gronwall_applied}
Let $A$, $G$ be real, continuous functions on $[0,T]$  with $G\geq 0$ and $\al, \be$ nonnegative constants. We suppose that on $[0,T]$ we have the inequality
\[
A(t) \leq A(0) + \int_0^t (\al + \be A(s)) G(s) \ud s\,.
\]
Then
\[
A(t) \leq A(0) + \left(\al + (A(0) + \al N) \be e^{\be N}\right) \int_0^t G(s) \ud s
\]
holds in $[0,T]$ with $N = \int_0^T G(t) \ud t$.
\end{corollary}

\begin{proof}
Apply the Gronwall inequality with $[a,b]=[0,T]$, $\Ps(t) = A(0) + \al \int_0^t G(s) \ud s$ and $\Ph(s) = \be G(s)$, and note that $G(s) \geq 0$ implies $\int_s^t G(u) \ud u \leq N$.
\end{proof}

\subsection{Continuous Riemannian Metrics}

A Riemannian metric on an (infinite-dimensional) manifold $M$ is a smooth, symmetric, bilinear, non-degenerate map
\[
g : TM \x_M TM \to \R\,.
\]
The induced geodesic distance is defined as
\[
\on{dist}(x,y) = \inf \left\{ L_g(\ga) \,:\, \ga(0) = x,\, \ga(1) = y,\, \ga \text{ piecewise smooth} \right\}\,,
\]
where, using $|v|_x = \sqrt{g_x(v,v)}$ to denote the induced norm,
\[
L_g(\ga) = \int_0^1 |\dot \ga(t)|_{\ga(t)} \ud t
\]
is the length of a path. We shall denote by $B(x,r)$ the open metric ball with respect to the geodesic distance,
\begin{align*}
B(x,r) &= \left\{ y \,:\, \on{dist}(x,y) < r \right\} \\
&= \left\{ \ga(1) \,:\, \ga(0) = x,\, L_g(\ga) < r \right\}\,.
\end{align*}

For some statements about the geodesic distance it is only necessary for $g$ to be a continuous Riemnannian metric; smoothness is not required. To be precise, we call $g$ \emph{weakly continuous}, if the map
\[
g : TM \x_M TM \to \R
\]
is continuous. This is to be contrasted with \emph{strong continuity}, which requires
\[
g : M \to L^2_{\on{sym}}(TM)
\]
to be a continuous section. Continuous Riemannian metrics and their induced geo\-de\-sic distance have been studied in finite dimensions in \cite{Burtscher2013_preprint}.

\subsection{Notation}

We will write
\[
f \lesssim_{A} g
\]
if there exists a constant $C > 0$, possibly depending on $A$, such that the inequality $f \leq C g$ holds.

For a smooth map $F$ from $\mc I^n(S^1,\R^d)$ or $\on{Imm}(S^1,\R^d)$ to any convenient vector space we denote by
\[
D_{c,h} F = \left.\frac{\ud}{\ud t}\right|_{t=0} F(c+th)
\]
the variation in the direction $h$.

\section{Estimates for the Geodesic Distance}
\label{sec:estimates}

In this section we prove estimates relating to the geodesic distance of Riemnnian metrics that are sufficiently strong. The main result will be Prop.~\ref{prop:sobolev_uniform_equivalent} showing that the ambient $H^n(d\th)$-norm and a Sobolev metric of order $n$ are equivalent with uniform constants on metric balls. This section extends the results of \cite{Bruveris2014} from plane curves to curves in $\R^d$.

We will make the following assumption on the Riemannian metric $G$ on the space $\on{Imm}(S^1,\R^d)$ for the rest of the section.

\begin{equation} \tag{$H_n$}
\label{hyp:hyp_compare_sobolev}
\begin{minipage}{0.9\textwidth}
Given a metric ball $B(c_0,r)$ in $\on{Imm}(S^1,\R^d)$, there exists a constant $C$, such that
\[
\| h \|^2_{H^n(ds)} = \int_{S^1} |h|^2 + |D_s^n h|^2 \ud s \leq C G_c(h,h)
\]
holds for all $c \in B(c_0,r)$ and all $h \in T_c\on{Imm}(S^1,\R^d)$.
\end{minipage}
\end{equation}

Note in particular that the class of metrics satisfying \eqref{hyp:hyp_compare_sobolev} includes Sobolev metrics with constant coefficients. Furthermore, Lem.~\ref{lem:poincare_inequalities} shows that if the metric $G$ satisfies \eqref{hyp:hyp_compare_sobolev}, then it also satisfies $(H_k)$ with $k \leq n$. To simplify the exposition we will work with smooth curves for now and extend the results to Sobolev immersions in Rem.~\ref{rem:extend_results}. First we collect some results from \cite{Bruveris2014}.

\begin{proposition}
\label{prop:lipschitz_basic_results}
Let $n \geq 2$ and $G$ be a weakly continuous Riemannian metric on $\on{Imm}(S^1,\R^d)$ satisfying \eqref{hyp:hyp_compare_sobolev}.
Then the following functions are continuous and Lipschitz continuous on every metric ball,
\begin{align*}
\log |c'| &: \big(\on{Imm}(S^1,\R^d), \on{dist}^G\big) \to L^\infty(S^1,\R)\,, \\
\ell_c^{1/2}, \ell_c^{-1/2} &: \big(\on{Imm}(S^1,\R^d), \on{dist}^G\big) \to \R_{>0}\,.
\end{align*}
In particular the following expressions are bounded on every metric ball
\[
\| c' \|_{L^\infty}, \left\| |c'|\inv \right\|_{L^\infty}, \ell_c, \ell_c\inv\,.
\]
Furthermore, the norms $L^2(d\th)$ and $L^2(ds)$ are uniformly equivalent on every metric ball, i.e., given a metric ball $B(c_0,r)$ in $\on{Imm}(S^1,\R^d)$, there exists a constant $C$, such that
\[
C\inv \| h \|_{L^2(d\th)} \leq \| h \|_{L^2(ds)} \leq C \| h \|_{L^2(d\th)}
\] 
holds for all $c \in B(c_0,r)$ and all $h \in L^2(S^1)$.
\end{proposition}

\begin{proof}
The Lipschitz continuity of $\ell_c^{1/2}$ and $\ell_c^{-1/2}$ is shown in \cite[Cor. 4.2]{Bruveris2014} and \cite[Lem. 4.4]{Bruveris2014} and the Lipschitz continuity of $\log |c'|$ in \cite[Lem. 4.10]{Bruveris2014}. The results there are formulated under slightly more restrictive hypotheses: it is assumed that $G$ is globally stronger than the $H^n(ds)$-norm with a constant, that does not depend on the choice of a metric ball and that $d=2$, i.e., for plane curves. Since all the arguments only consider paths, that lie in some metric ball, the constant $C$ in~\eqref{hyp:hyp_compare_sobolev} can also depend on the ball and the variational formulae in these proofs are valid for curves in $\R^d$ without a change. The equivalence of the norms $L^2(d\th)$ and $L^2(ds)$ follows from
\[
\left( \min_{\th \in S^1} |c'(\th)| \right) \| h \|^2_{L^2(d\th)}
\leq \| h \|^2_{L^2(ds)} \leq
\| c'\|_{L^\infty} \| h \|^2_{L^2(d\th)}\,,
\]
and the boundedness of $|c'(\th)|$ from above and below on a metric ball.
\end{proof}

The following lemma encapsulates a general principle for proving the Lipschitz continuity of functions with respect to the geodesic distance.

\begin{lemma}
\label{lem:loc_lipschitz_result}
Let $(M,g)$ be a Riemannian manifold with a weakly continuous metric and $f : M \to F$ a $C^1$-function into a normed space $F$. Assume that for each metric ball $B(y,r)$ in $M$ there exists a constant $C$, such that
\begin{equation}
\label{eq:derivative_bound}
| T_x f.v |_F \leq C \left(1 + |f(x)|_F\right) |v|_x
\end{equation}
holds for all $x \in B(y,r)$ and all $v \in T_x M$. Then the function
\[
f : (M, d) \to (F, |\cdot|_F)
\]
is continuous and Lipschitz continuous on every metric ball. In particular $f$ is bounded on every metric ball.

If the constant $C$ can be chosen such that \eqref{eq:derivative_bound} holds globally for $x \in M$, then $f$ is globally Lipschitz continuous.
\end{lemma}

By carefully following the proof, it is possible to find explicit values for the Lipschitz constant. We will not need the explicit values and so we only note that the Lipschitz constant of $f$ on the ball $B(y,r)$ will depend on the constant $C$ for the ball $B(y,3r)$.

\begin{proof}
Fix a metric ball $B(y,r)$ and two points $x_1, x_2 \in B(y,r)$. Then $d(x_1,x_2) < 2r$ and we can choose a piecewise smooth path $x(t)$ connecting $x_1$ and $x_2$ with $L_g(x) < 2r$. Then $d(y, x(t)) < 3r$ and thus the path $x$ remains within a metric ball of radius $3r$ around $y$.

Starting from
\[
f(x(t)) - f(x_1) = \int_0^t T_{x(\ta)} f.\dot x(\ta) \ud \ta\,,
\]
we obtain
\[
\left|f(x(t)) - f(x_1)\right|_F \leq \int_0^t \left| T_{x(\ta)} f.\dot x(\ta)\right|_F \ud \ta \lesssim_{y,r}
\int_0^t \big(1 + \left|f(x(\ta))\right|_F\big) |\dot x(\ta)|_{x(\ta)} \ud \ta\,,
\]
and by setting
\[
A(t) = \left|f(x(t)) - f(x_1)\right|_F\,,
\]
we can rewrite the above inequality as
\[
A(t) \lesssim_{y,r} \int_0^t \big(1 + |f(x_1)|_F + A(t)\big) |\dot x(\ta)|_{x(\ta)} \ud \ta\,.
\]
Using Gronwall's inequality Cor.~\ref{cor:gronwall_applied} this leads to
\[
A(t) \lesssim_{y,r} \big( 1 + |f(x_1)|_F \big) \int_0^t |\dot x(\ta)|_{x(\ta)} \ud \ta
\leq \big( 1 + |f(x_1)|_F \big) L_g(x)\,.
\]
Taking the infimum over all paths $x$ between $x_1$ and $x_2$ we obtain almost the required inequality,
\[
\left|f(x_1) - f(x_2)\right|_F \lesssim_{y,r} ( 1 + |f(x_1)|_F )\, d(x_1, x_2)\,.
\]
To remove the dependence on $|f(x_1)|_F$ on the right hand side, we use the inequality with $x_2=y$ as follows,
\[
|f(x_1)|_F \leq |f(y)|_F + |f(x_1) - f(y)|_F
\lesssim_{y,r} (1 + r) |f(y)|_F + r\,.
\]
This concludes the proof.
\end{proof}

The next lemma is a preparation to prove Prop.~\ref{prop:things_lipschitz}. We need to calculate the variations of $D_s^k c$ and $D_s^k |c'|$. In fact we are only interested in the terms of highest order and collect the rest in the polynomials $P$ and $Q$.

\begin{lemma}
\label{lem:variation_Dsk_c}
Let $c \in \on{Imm}(S^1,\R^d)$, $h \in T_c \on{Imm}(S^1,\R^d)$ and $k \geq 0$. Then
\begin{align*}
D_{c,h} \left( D_s ^k c \right) &= D_s^k h - \langle D_s^k h, v \rangle v - k \langle D_s h, v \rangle D_s^k c - \langle D_s h, D_s^k c \rangle v +{} \\
&\qquad\qquad+ P(D_s c, \dots, D_s^{k-1}c; D_s h,\dots,D_s^{k-1} h) \\
D_{c,h} \left( D_s^k |c'| \right) &= \langle D_s^{k+1} h, v \rangle |c'|
- (k-1) \langle D_s h, v \rangle D_s^k |c'| + \langle D_s h, D_s^{k+1} c \rangle |c'| +{} \\
&\qquad\qquad+ Q(|c'|,\dots,D_s^{k-1}|c'|, D_s c, \dots,D_s^k c; D_s h,\dots
D_s^k h)
\end{align*}
and $P(\dots)$ and $Q(\dots)$ are polynomials in the respective variables and linear in the components of $D_s h, \dots, D_s^k h$.
\end{lemma}

\begin{proof}
We have
\[
D_{c,h}(D_s^k) = - \sum_{j=0}^{k-1} D_s^j \o \langle D_sh,v \rangle \o D_s^{k-j}\,,
\]
and thus
\[
D_{c,h} \left(D_s^k c \right) = D_s^k h - \sum_{j=0}^{k-1} D_s^j \big( \langle D_s h, v\rangle D_s^{k-j} c \big)\,.
\]
Next we use the identity \cite[(26.3.7)]{NIST2010},
\[
\sum_{j=i}^{k-1} \binom{j}{i} = \binom{k}{i+1} \,,
\]
and the product rule for differentiation to obtain
\begin{align*}
D_{c,h} \left(D_s^k c \right) &= D_s^k h - \sum_{j=0}^{k-1} \sum_{i=0}^j \binom{j}{i} D_s^i \langle D_s h, v\rangle D_s^{k-j+j-i} c \\
&= D_s^k h - \sum_{i=0}^{k-1} \sum_{j=i}^{k-1} \binom{j}{i} D_s^i \langle D_s h, v\rangle D_s^{k-i} c\\
&= D_s^k h - \sum_{i=0}^{k-1} \binom{k}{i+1} D_s^i \langle D_s h, v\rangle D_s^{k-i} c\,.
\end{align*}
It is clear that the expression is linear in $h$. It remains to isolate the terms involving derivatives of order $k$. These are
\begin{multline*}
D_s^k h - \binom{k}{1} \langle D_s h, v \rangle D_s^k c - \binom{k}{k} \langle D_s^k h, v \rangle D_s c - \binom{k}{k} \langle D_s h, D_s^{k-1} v \rangle D_s c = \\
= D_s^k h - \langle D_s^k h, v \rangle v - k \langle D_s h, v \rangle D_s^k c -
\langle D_s h, D_s^k c\rangle D_s c\,,
\end{multline*}
thus proving the first formula. For the second one we have
\begin{align*}
D_{c,h} \left( D_s^k |c'| \right) &= D_s^k \left( \langle D_s h, v \rangle |c'|\right)
- \sum_{j=0}^{k-1} D_s^j \big( \langle D_s h, v\rangle D_s^{k-j} |c'| \big) \\
&= D_s^k \left( \langle D_s h, v \rangle |c'|\right) -
\sum_{i=0}^{k-1} \binom{k}{i+1} D_s^i \langle D_s h, v\rangle D_s^{k-i} |c'|\,.
\end{align*}
The terms involving $k+1$ derivatives are
\[
\langle D_s^{k+1} h, v \rangle |c'| + \langle D_s h, D_s^{k+1} c \rangle |c'|
+ \langle D_s h, v\rangle D_s^k|c'| - k \langle D_s h, v \rangle D_s^k |c'|
\]
and the remaining terms can be collected in the polynomial $Q(\dots)$.
\end{proof}

This result can be seen as the generalization of \cite[Thm. 4.7]{Bruveris2014} to curves in $\R^d$. It is the main tool to prove Prop.~\ref{prop:sobolev_uniform_equivalent}.

\begin{proposition}
\label{prop:things_lipschitz}
Let $n \geq 2$ and $G$ be a weakly continuous Riemannian metric on $\on{Imm}(S^1,\R^d)$ satisfying \eqref{hyp:hyp_compare_sobolev}. 

Then the following functions are continuous and Lipschitz continuous on every metric ball,
\begin{align*}
D_s^k c &: \big(\on{Imm}(S^1,\R^d), \on{dist}^G\big) \to L^2(S^1,\R^d)\,,
&& 0 \leq k \leq n\,, \\
D_s^k c &: \big(\on{Imm}(S^1,\R^d), \on{dist}^G\big) \to L^\infty(S^1,\R^d)\,,
&& 0 \leq k \leq n-1\,, \\
D_s^k |c'| &: \big(\on{Imm}(S^1,\R^d), \on{dist}^G\big) \to L^2(S^1,\R)\,,
&& 0 \leq k \leq n-1\,, \\
D_s^k |c'| &: \big(\on{Imm}(S^1,\R^d), \on{dist}^G\big) \to L^\infty(S^1,\R)\,,
&& 0 \leq k \leq n-2\,.
\end{align*}
In particular the following expressions
\begin{align*}
\| c \|_{L^\infty}, \dots, \| D_s^{n-1} c \|_{L^\infty}, \big\| |c'| \big\|_{L^\infty},
\dots, \big\| D_s^{n-2} |c'| \big\|_{L^\infty} \\
\| D_s^n c \|_{L^2(d\th)}, \| D_s^n c \|_{L^2(ds)}, \| c \|_{H^n(ds)}, \big \| D_s^{n-1} |c'| \big\|_{L^2(d\th)}
\end{align*}
are bounded on every metric ball. 
\end{proposition}

\begin{proof}
Fix a metric ball $B(c_0,r)$. We will use Lem.~\ref{lem:loc_lipschitz_result} to establish the proposition. Let us start with the Lipschitz continuity of $D_s^k c$ in the $L^\infty$-norm. We fix $n$ and proceed via induction on $k$. For $k=0$ we have $D_{c,h} c = h$ and via
\[
\|h\|_{L^\infty} \lesssim_{d,R} \| h \|_{L^2(ds)} + \| D_s h \|_{L^2(ds)} \lesssim_{d,R} \sqrt{G_c(h,h)}\,,
\]
we are done. For $k=1$ we similarly have
\[
D_{c,h} (D_s c) = D_s h - \langle D_s h, v \rangle v\,,
\]
and
\[
\| D_s h - \langle D_s h, v \rangle v  \|_{L^\infty} \leq \| D_s h \|_{L^\infty} \lesssim_{d,r} \| D_s^2 h \|_{L^2(ds)} \lesssim_{d,R} \sqrt{G_c(h,h)}\,.
\]
For the induction step assume $2 \leq k \leq n-1$ and that the result has been established for $k-1$. Then $\| D_s^j h\|_{L^\infty}$ is bounded on metric balls for $0 \leq j \leq k-1$ and we can estimate using Lem.~\ref{lem:variation_Dsk_c},
\begin{align*}
\left\| D_{c,h} \left( D_s ^k c \right) \right\|_{L^\infty} &\leq 
\left\| D_s^k h \right\|_{L^\infty} + (k+1) \left\| D_s h \right\|_{L^\infty} \left\| D_s^k c \right\|_{L^\infty} +{} \\
&\qquad\qquad+ \left\| P(D_s c, \dots, D_s^{k-1}c; D_s h,\dots,D_s^{k-1} h)\right\|_{L^\infty} \\
&\lesssim_{d,R} \left( 1 + \| D_s^k c \|_{L^\infty} \right) \sqrt{G_c(h,h)}\,,
\end{align*}
since $P$ is linear in $h$. Via Lem.~\ref{lem:loc_lipschitz_result} this concludes the proof of the $L^\infty$-continuity of $D_s^k c$.

Next we show the $L^2(d\th)$-continuity of $D_s^k c$ for $0 \leq k \leq n$. Again via Lem.~\ref{lem:variation_Dsk_c} we have
\begin{align*}
\left\| D_{c,h} \left( D_s ^k c \right) \right\|_{L^2(d\th)} &\leq 
\left\| D_s^k h \right\|_{L^2(d\th)} + (k+1) \left\| D_s h \right\|_{L^\infty} \left\| D_s^k c \right\|_{L^2(d\th)} +{} \\
&\qquad\qquad+ \left\| P(D_s c, \dots, D_s^{k-1}c; D_s h,\dots,D_s^{k-1} h)\right\|_{L^2(d\th)} 
\end{align*}
Since $c, D_s c, \dots,  D_s^{n-1} c$ are bounded in the $L^\infty$-norm, we can bound $P(\dots)$ in the $L^\infty$-norm by
\[
\left\| P(D_s c, \dots, D_s^{k-1}c; D_s h,\dots,D_s^{k-1} h)\right\|_{L^\infty}
\lesssim_{d,R} \sqrt{G_c(h,h)}\,,
\]
and thus
\[
\left\| D_{c,h} \left( D_s ^k c \right) \right\|_{L^2(d\th)} \lesssim_{d,R} \left( 1 + \| D_s^k c \|_{L^2(d\th)} \right) \sqrt{G_c(h,h)}\,.
\]
Now apply Lem.~\ref{lem:loc_lipschitz_result}.

The Lipschitz continuity of $D_s^k|c'|$ in the $L^\infty$- and $L^2(d\th)$-norms can be shown in exactly the same way, using the second part of Lem.~\ref{lem:variation_Dsk_c}; note that since $k \leq n-1$, all the terms involving $D_s c,\dots D_s^k c$ in $Q(\dots)$ are bounded on metric balls in the $L^\infty$-norm and thus can be effectively ignored.
\end{proof}

This is the main result of the section and it will be essential to show metric completeness of Sobolev metrics.

\begin{proposition}
\label{prop:sobolev_uniform_equivalent}
Let $n \geq 2$ and $G$ be a weakly continuous Riemannian metric on $\on{Imm}(S^1,\R^d)$ satisfying \eqref{hyp:hyp_compare_sobolev}.

Then, given a metric ball $B(c_0,r)$ in $\on{Imm}(S^1,\R^d)$, there exists a constant $C$, such that
\[
C\inv \| h \|_{H^n(d\th)} \leq \| h \|_{H^n(ds)} \leq C \| h \|_{H^n(d\th)}
\]
holds for all $c \in B(c_0,r)$ and all $h \in H^n(S^1,\R)$.
\end{proposition}

The proof of this proposition can be found in \cite[Lem. 5.1]{Bruveris2014} for plane curves. The proof can be reused without change for curves in $\R^d$, if we refer to Prop.~\ref{prop:things_lipschitz} to obtain boundedness of $D_s^k |c'|$ on metric balls, where necessary. 

If $G$ is a Sobolev metric of order $n \geq 2$ with constant coefficients, then Lem.~\ref{lem:poincare_inequalities} shows that the norm induced by $G_c(\cdot,\cdot)$ is equivalent to the $H^n(ds)$-norm with a uniform constant, i.e., there exists $C_1$, such that
\[
C_1\inv \| h \|_{H^n(ds)} \leq \sqrt{G_c(h,h)} \leq C_1 \| h \|_{H^n(ds)}
\]
holds for all $c \in \on{Imm}(S^1,\R^d)$ and all $h \in H^n(S^1,\R)$. Hence the the norm induced by $G_c(\cdot,\cdot)$ is also equivalent to the ambient $H^n(d\th)$-norm with uniform constants on every metric ball.

\begin{remark}
\label{rem:extend_results}
Let $n \geq 2$ and $G$ be a weakly continuous metric on $\on{Imm}(S^1,\R^d)$. If $G$ can be extended to a weakly continuous Riemannian metric on $\mc I^n(S^1,\R^d)$, then the statements of this section can also be extended from $\on{Imm}(S^1,\R^d)$ to $\mc I^n(S^1,\R^d)$. This is true for Prop.~\ref{prop:lipschitz_basic_results}, the calculations in Lem.~\ref{lem:variation_Dsk_c}, Prop.~\ref{prop:things_lipschitz} and Prop.~\ref{prop:sobolev_uniform_equivalent}. Consider for example the inequality
\[
\| D^n_{c_1} c_1 - D^n_{c_2} c_2 \|_{L^2(d\th)} \leq C \on{dist}(c_1,c_2)\,,
\]
from Prop.~\ref{prop:things_lipschitz}, valid for $c_1, c_2 \in \on{Imm}(S^1,\R^d)$ in a bounded metric ball. Here $\on{dist}$ is the geodesic distance on $(\on{Imm}(S^1,\R^d), G)$. Proposition~\ref{prop:geod_distance_submanifold} and Rem.~\ref{rem:construct_P_by_hand} show that the geodesic distance on $\mc I^n(S^1,\R^d)$ restricted to $\on{Imm}(S^1,\R^d)$ coincides with the geodesic distance on $\on{Imm}(S^1,\R^d)$. Given $c_1, c_2 \in \mc I^n(S^1,\R^d)$, choose sequences of smooth immersions $c_1^j, c_2^j \in \on{Imm}(S^1,\R^d)$ with $c_i^j \to c_i$ in $\mc I^n(S^1,\R^d)$. Then $\on{dist}(c_1^j, c_2^j) \to \on{dist}(c_1, c_2)$, because the metric topology is weaker than the manifold topology. The left hand side also converges, because $c \mapsto D_s^n c$ is a continuous map $\mc I^n(S^1,\R^d) \to L^2(S^1,\R^d)$. Thus the inequality continues to hold on metric balls in $\mc I^n(S^1,\R^d)$.
\end{remark}

\section{Metric and Geodesic Completeness}
\label{sec:met_geod_complete}

\subsection{Space of Sobolev Immersions}

The estimates of the previous section allow us to find upper and lower bounds for the geodesic distance on $\mc I^n(S^1,\R^d)$.

\begin{lemma} 
\label{lem:dist_uniform_equivalence}
Let $n \geq 2$ and $G$ be a Sobolev metric of order $n$ with constant coefficients. Then
\begin{enumerate}
\item
Given a metric ball $B(c_0,r)$ in $\mc I^n(S^1,\R^d)$, there exists $C$, such that
\[
\| c_1 - c_2 \|_{H^n(d\th)} \leq C \on{dist}(c_1, c_2)\,,
\]
holds for all $c_1, c_2 \in B(c_0,r)$.
\item
Given $c_0 \in \mc I^n(S^1,\R^m)$, there exist $r>0$ and $C$, such that
\[
\on{dist}(c_1, c_2) \leq C \| c_1 - c_2 \|_{H^n(d\th)}\,,
\]
holds for all $c_1, c_2 \in B(c_0,r)$.
\end{enumerate}
\end{lemma}

\begin{proof}
Given $c_1, c_2 \in B(c_0,r)$, let $c(t, \th)$ be a piecewise smooth path of length $L(c) < r$ connecting them. Then,
\[
\| c_1 - c_2 \|_{H^n(d\th)} \leq \int_0^1 \| \dot c(t) \|_{H^n(d\th)} \ud t \leq
C \int_0^1 \sqrt{G_c(\dot c, \dot c)} \ud t \leq C L(c)\,,
\]
where $C$ is given by Prop.~\ref{prop:sobolev_uniform_equivalent} and depends only on $c_0$ and $r$. By taking the infimum over all paths we obtain the first part of the statement.

Given $c_0 \in \mc I^n(S^1,\R^d)$, let $U$ be a convex, open neighborhood of $c_0$ in $\mc I^n(S^1,\R^d)$ and $r > 0$, such that $B(c_0,r) \subseteq U$. Such an $r$ exists, because $G$ is a smooth, strong Riemannian metric and hence the geodesic distance induces the manifold topology, see \cite[Prop. 6.1]{Lang1999}. Given $c_1, c_2 \in B(c_0,r)$, define the path $c(t) = c_1 + t(c_2 - c_1)$ to be the linear interpolation between $c_1$ and $c_2$. Then,
\[
\on{dist}(c_1, c_2) \leq L(c) = \int_0^1 \sqrt{G_c(c_2 - c_1, c_2 - c_1)} \ud t
\leq C \| c_2 - c_1 \|_{H^n(d\th)}\,,
\]
with $C$ again given by Prop.~\ref{prop:sobolev_uniform_equivalent}. This proves the second part.
\end{proof}

The lemma shows that the identity map
\[
\on{Id} : (\mc I^n(S^1,\R^d), \on{dist}) \to (\mc I^n(S^1,\R^d), \| \cdot \|_{H^n(d\th)})
\]
is locally bi-Lipschitz. This is sufficient to show the metric completeness of the space $(\mc I^n(S^1,\R^d), G)$.

\begin{theorem}
\label{thm:metric_completeness}
Let $n \geq 2$ and $G$ be a Sobolev metric of order $n$ with constant coefficients. Then
\begin{enumerate}
\item
$\left(\mc I^n(S^1,\R^d), \on{dist}\right)$ is a complete metric space;
\item
$\left(\mc I^n(S^1,\R^d), G\right)$ is geodesically complete.
\end{enumerate}
\end{theorem}

\begin{proof}
Let $(c^j)_{j \in \mb N}$ be a Cauchy sequence with respect to the geodesic distance. Then the sequence remains within a bounded metric ball in $\mc I^n(S^1,\R^d)$ and by Lem.~\ref{lem:dist_uniform_equivalence} it is also a Cauchy sequence with respect to $\| \cdot \|_{H^n(d\th)}$. As $H^n(S^1,\R^d)$ is complete, there exists a limit $c^\ast \in H^n(S^1,\R^d)$ and $\| c^j - c^\ast \|_{H^n(d\th)} \to 0$. From Prop.~\ref{prop:lipschitz_basic_results} we see that $\| \p_\th c^j(\th)\| \geq C > 0$ is bounded from below, away from 0, on metric balls and thus, so is the limit; in particular $c^\ast \in \mc I^n(S^1,\R^d)$. Finally, the second part of Lem.~\ref{lem:dist_uniform_equivalence} shows that $\on{dist}(c^j, c^\ast) \to 0$. Hence $(\mc I^n(S^1,\R^d), \on{dist})$ is complete.

It is shown in \cite[Sect.~3]{Bruveris2014} that Sobolev metrics of order $n \geq 2$ are smooth on $\mc I^n(S^1,\R^d)$ and \cite[Prop. 6.5]{Lang1999} shows that on a strong Riemannian manifold metric completeness implies geodesic completeness.
\end{proof}

A direct proof of geodesic completeness for plane curves can be found in \cite{Bruveris2014}. In the next section we will prove the third completeness statement, the existence of minimizing geodesics between any two curves.

\subsection{Space of Smooth Immersions}

Of course one can also consider Sobolev metrics on the space $\on{Imm}(S^1,\R^d)$ of smooth immersions. In this case we do not have metric completeness, but interestingly enough the space $(\on{Imm}(S^1,\R^d), G)$ is geodesically complete. We are nevertheless able to identify the metric completion of $\on{Imm}(S^1,\R^d)$. That the metric completion of $\on{Imm}(S^1,\R^d)$ equals $\mc I^n(S^1,\R^d)$ for plane curves was remarked in \cite{Bauer2014_preprint} using the same method as below.

\begin{theorem}
\label{thm:immersion_weak_completion}
Let $n \geq 2$ and $G$ be a Sobolev metric of order $n$ with constant coefficients. For $m > n$ we have:
\begin{enumerate}
\item
The geodesic distance on $(\mc I^m, G)$ coincides with the restriction of the geodesic distance on $(\mc I^n, G)$ to $\mc I^m$. In particular the metric completion of $(\mc I^m, \on{dist}^n)$ is $(\mc I^n, \on{dist}^n)$.
\item
$(\mc I^m, G)$ is geodesically complete.
\end{enumerate}
The same holds for $m=\infty$, i.e., the space $\on{Imm}$ of smooth immersions.
\end{theorem}

\begin{proof}
Let $m > n$ or $m=\infty$. Then $\mc I^m$ is a dense, weak submanifold of $\mc I^n$ and thus by Prop.~\ref{prop:geod_distance_submanifold} the restriction of the geodesic distance on $\mc I^n$ coincides with the geodesic distance on $\mc I^m$. In particular the notation $(\mc I^m, \on{dist})$ is unambiguous. The metric space $(\mc I^n, \on{dist})$ is complete by Thm.~\ref{thm:metric_completeness} and thus it is the metric completion of $(\mc I^m, \on{dist})$. For $m=\infty$ we need to use Rem.~\ref{rem:construct_P_by_hand} and one can choose the sequence of operators $P_j : H^n \to C^\infty$, for example, to be convolution with mollifiers, see, e.g., \cite[Sect. 2.28]{Adams2003}.

Geodesic completeness of $(\mc I^m, G)$ follows from the property of the geodesic equation to preserve the smoothness of the initial conditions. Thus, given $(c_0, u_0) \in T\mc I^m$, the corresponding geodesic $c(t)$ exists for all time in $\mc I^n$ and by \cite[Thm. 3.7]{Bruveris2014}, which remains valid for curves in $\R^d$, we have $(c(t), \dot c(t)) \in T \mc I^m$ for all $t > 0$. See also \cite[Thm. 12.1]{Ebin1970b} and \cite[App. A]{Bauer2014b} for more details on why the geodesic equation preserves the smoothness of initial conditions.
\end{proof}

\section{Existence of Minimizing Geodesics}
\label{sec:min_geodesics}

\subsection{Space of Sobolev Immersions}

In this section we will show that any two curves in the same connected component of $\mc I^n(S^1,\R^d)$ can be joined by a minimizing geodesic with respect to a Sobolev metric of order $n \geq 2$ with constant coefficients. See Sect.~\ref{sec:connectivity} for a discussion of the connectivity of $\mc I^n(S^1,\R^d)$.

We will denote in this section the unit interval by $I = [0,1]$. To shorten notaion we set
\[
H^1_t H^n_\th = H^1_tH^n_\th([0,1] \x S^1, \R^d) \cong H^1(I, H^n(S^1,\R^d))\,,
\]
and similarly for $C_t H^n_\th$, $L^2_t L^2_\th$, etc.

\begin{theorem}
\label{thm:ex_of_minimizers}
Let $n \geq 2$ and $G$ be a Sobolev metric of order $n$ with constant coefficients. Given $c_0 \in \mc I^n(S^1,\R^d)$ and a weakly closed set $A \subseteq \mc I^n(S^1,\R^d)$, such that at least one curve in $A$ belongs to the same connected component as $c_0$, there exists a geodesic realizing the minimal distance between $c_0$ and $A$.
\end{theorem}

To restate the theorem, given $c_0$ and $A$, there exists $c_1 \in A$ and a geodesic $c(t)$ with $c(0) = c_0$ and $c(1) = c_1$, such that
\[
L(c) = \on{dist}(c_0, c_1) = \on{dist}(c_0, A) = \inf_{\tilde c \in A} \on{dist}(c_0, \tilde c)\,,
\]
and the same holds for the energy $E(c)$ and the squared distance.

Before we proceed with the proof of the theorem, which will be a bit technical, we would like to comment on possible generalizations of the result.

\subsection{Metrics with Non-Constant Coefficients}

Sobolev metrics with non-con\-stant coefficients have been of interest; for example \cite{Shah2013, Bauer2014b} look at second order metrics and \cite{Mennucci2008} at metrics of higher order. Similarly length-weighted metrics are studied in \cite{Bauer2012d}. 

\begin{remark}
The proof of Thm.~\ref{thm:ex_of_minimizers} continues to work in a slightly more general setting. We need that $G$ is a continuous Riemannian metric on $\mc I^n(S^1,\R^d)$, that is uniformly bounded and uniformly coercive with respect to the background $H^n(d\th)$-norm on every metric ball. This is necessary to show that a minimizing sequence is bounded in the Hilbert space $H^1(I, H^n(d\th))$. The condition $n\geq 2$ is necessary to show that weak limits still satisfy $|c'(t,\th)| > 0$. In fact $n > 3/2$ would be sufficient here. 

Finally, to show that the energy $E$ is sequentially weakly lower semicontinuous we used special properties of the arc length derivative, established in Lem.~\ref{lem:Ds_weak_convergence}. The same argument works, if the metric $G$ is of the form
\[
G_c(h,h) = \sum_{i=1}^N \| A_i(c)h\|^2_{F_i}
\]
with some Hilbert spaces $F_i$ and smooth maps $A_i : \mc I^n \to L(H^n, F_i)$, and the maps $A_i$ have the following property:
\begin{equation}
\label{eq:A_convergence_prop}
\begin{array}{c}
c^j \to c \text{ weakly in } H^1_t \mc I^n_\th \\
(c^j)_{j \in \mb N} \text{ bounded in } H^1_t H^n_\th
\end{array}
\Rightarrow
A_i(c^j)\dot c^j \rightharpoonup A_i(c) \dot c \text{ weakly in } L^2(I, F_i)\,.
\end{equation}
The proof can then be reused without change.
\end{remark}

This remark allows us to consider Sobolev metrics with non-constant coefficients, for example the curvature weighted metric of order 3,
\[
G_c(h,h) = \int_{S^1} (1 + \ka^2) (|h|^2 + |D_s^3 h|^2) \ud s\,,
\]
or the length weighted metric of order 2
\[
G_c(h,h) = \int_{S^1} \frac{2\pi}{\ell_c} |h|^2 + \left(\frac{\ell_c}{2\pi}\right)^3 |D_s^2 h|^2 \ud s\,.
\]
The latter metric has the property, that it is constant on curves, which are pa\-ra\-met\-rized by constant speed; that is, if $c \in \mc I^2$ with $|c'|\equiv \text{const.}$, then $|c'| = \ell_c/2\pi$ and
\[
G_c(h,h) = \int_{S^1} |h|^2 + |h''|^2 \ud \th\,.
\]
We see that the right hand side is independent of $c$. However the uniform boundedness and uniform coercivity for this metric do not follow immediately from the results in Sect.~\ref{sec:estimates}, since it is not clear that $G$ satisfies hypothesis \eqref{hyp:hyp_compare_sobolev}.

\begin{remark}
A related existence result is presented in \cite{Rumpf2012_preprint}. There the authors assume that $g : U \to L^2_{\on{sym}}(E)$ is a continuous Riemannian metric, with $U$ being an open subset of a Hilbert space $E$, uniformly bounded and coercive with respect to the background metric. With regard to continuity they make the following stronger assumption: let $F$ be another Hilbert space and the embedding $E \hookrightarrow F$ compact; then $g$ should be continuous with respect to the topology of $F$.

While we cannot use this result by itself, since the functional $\int_{S^1} |D_s^n h|^2 \ud s$ is not continuous in a weaker topology than the $H^n(d\th)$ topology, the above result permits us to add lower order terms to the metric. The embedding $H^n \hookrightarrow H^{n-1}$ is compact and so we are free to add metrics, that are continuous on $\mc I^{n-1}$ to $G$, without having to worry, whether they are of the specific form to satisfy \eqref{eq:A_convergence_prop}.
\end{remark}

\subsection{Weak Convergence of Arc Length Derivatives}

To prove Thm.~\ref{thm:ex_of_minimizers} we will need a result about the behaviour of arc length derivatives. This will be Lem.~\ref{lem:Ds_weak_convergence}. We start with two facts about smoothness of operations in Sobolev spaces. The first is a simple generalization of \cite[Prop. 2.20]{Inci2013}.

\begin{lemma}[Prop. 2.20, \cite{Inci2013}]
\label{lem:left_composition}
Let $M$ be a closed manifold, $s > \dim M/2$ and $g \in C_b^\infty(\R^d, \R^m)$. Then left-translation
\[
L_g : H^s(M, \R^d) \to H^s(M,\R^m)\,,\quad f \mapsto g \o f
\]
is a $C^\infty$-map.
\end{lemma}

We now apply this lemma to show that the term $|c'|\inv$, that appears in the arc-length derivative is well-behaved. We will need to apply the lemma with Sobolev spaces of non-integer order. To emphasize this we will use $s$ instead of $n$ for the Sobolev order.

\begin{lemma}
\label{lem:one_over_norm_cp_continuous}
Let $s \in \R$ and $s>3/2$. The map
\[
\mc I^s(S^1,\R^d) \to H^{s-1}(S^1,\R^d),\quad c \mapsto \frac{1}{|c'|}
\]
is smooth and bounded on sets with $\| c\|_{H^s_\th}$ bounded from above and $\inf_{\th \in S^1} |c'| > M$ for some $M>0$.
\end{lemma}

\begin{proof}
Let $U \subset \mc I^s(S^1,\R^d)$ be an open subset with $\| c\|_{H^s_\th}$ bounded from above and $\inf |c'|$ from below. Then we can extend the function $g(x) = 1/|x|$ to $g \in C_b^\infty(\R^d, \R)$, such that $g \o c'(\th) = 1/|c'(\th)|$ for all $c \in U$. The lemma now follows from Lem.~\ref{lem:left_composition}.
\end{proof}

And now the main lemma.

\begin{lemma}
\label{lem:Ds_weak_convergence}
Let $s \in \R$, $s > 3/2$ and $0 \leq k \leq s$. If $c^j, c \in H^1_t \mc I^s_\th$ and $h^j, h \in L^2_t H^k_\th$, then
\[
\begin{array}{c}
c^j \rightharpoonup c \text{ weakly in } H^1_t \mc I^s_\th \\
h^j \rightharpoonup h \text{ weakly in } L^2_t H^k_\th \\
(h^j)_{j \in \mb N} \text{ bounded in } L^2_t H^k_\th
\end{array}
\; \Rightarrow \;
D_{c^j}^k h^j \rightharpoonup D_c^k h \text{ weakly in } L^2_t L^2_\th\,.
\]
\end{lemma}

\begin{proof}
We will show that the above hypotheses imply
\[
\begin{array}{c}
D_{c^j} h^j \rightharpoonup D_c h \text{ weakly in } L^2_t H^{k-1}_\th
\text{ and } \\
(D_{c^j} h^j)_{j \in \mb N} \text{ is bounded in } L^2_t H^{k-1}_\th\,.
\end{array}
\]
The result then follows by induction. Let $\ep$ be such that $0 < \ep < 1$ and $s-\ep > 3/2$. Since a sequence converges against a limit, if every subsequence has a subsequence converging against that same limit, we are free to work with subsequences in our argument. The embedding $H^1_tH^s_\th \hookrightarrow C_t H^{s-\ep}_\th$ is compact, and so we can choose a subsequence of $(c^j)_{j \in \mb N}$, such that $c^j \to c$ in $C_t H^{s-\ep}_\th$.

The sequence $(h^j)_{j \in \mathbb N}$ is bounded, $L^2_tH^{2k-2}_\th$ is dense in $L^2_t H^k_\th$ and so by \cite[Thm. V.1.3]{Yosida1980} it is enough show that
\[
\langle D_{c^j} h^j - D_c h, u \rangle_{L^2_\th H^{k-1}_\th} \to 0
\]
for every $u \in L^2_t H^{2k-2}_\th$. Setting $w = u + (-1)^{k-1} \p_\th^{2k-2} u$ we have $w \in L^2_tL^2_\th$ and
\begin{align*}
\Big| &\langle D_{c^j} h^j - D_c h, u \rangle_{L^2_t H^{k-1}_\th} \Big| =
\Big| \langle D_{c^j} h^j - D_c h, w \rangle_{L^2_t L^2_\th} \Big| \leq \\
&\leq\left| \int_0^1 \int_{S^1} \left\langle 
\left( |\p_\th c^j|\inv - |\p_\th c|\inv\right) \p_\th h^j
+ |\p_\th c|\inv \left( \p_\th h^j - \p_\th h \right), 
w \right\rangle \ud \th \ud t \right| \\
&\leq \left\| |\p_\th c^j|\inv - |\p_\th c|\inv \right\|_{C_t C_\th}
\| \p_\th h^j \|_{L^2_tL^2_\th} \| w \|_{L^2_tL^2_\th} +
\left|\left\langle \p_\th h^j - \p_\th h, |\p_\th c|\inv w \right\rangle_{L^2_t L^2_\th}\right|\!.
\end{align*}
Using Lem.~\ref{lem:one_over_norm_cp_continuous} with $s-\ep$, since $I=[0,1]$ is compact, we obtain $|\p_\th c^j|\inv \to |\p_\th c|\inv$ not only pointwise in $t$, but uniformly, that is in $C_t H^{s-\ep-1}_\th$, and with the help of the Sobolev embedding $H^{s-\ep-1}_\th \hookrightarrow C_\th$ also in $C_t C_\th$. The term $\| \p_\th h^j \|_{L^2_tL^2_\th}$ is bounded because weakly convergent sequences are bounded. Since $\p_\th h^j \rightharpoonup \p_\th h$ weakly in $L^2_t H^{k-1}_\th$, we obtain
\[
\left\langle \p_\th h^j - \p_\th h, |\p_\th c|\inv w \right\rangle_{L^2_t L^2_\th} \to 0\,.
\]
This shows the required weak convergence.

The boundedness of $(D_{c^j} h^j)_{j \in \mb N}$ follows from the inequality
\[
\left\| |\p_\th c^j|\inv \p_\th h^j \right\|_{L^2_t H^{k-1}_\th} \leq
C \left\| |\p_\th c^j|\inv \right\|_{C_t H^{n-\ep}_\th}
\left\| \p_\th h^j \right\|_{L^2_t H^{k-1}_\th}\,.
\]
Because $c(t) \in \mc I^s$ and $c^j \to c$ in $C_tH^{s-\ep}_\th$, the set $\{ c^j(t) \,:\, (t, j) \in I \x \mb N\}$ clearly has $|c'|$ bounded from below and thus by Lem.~\ref{lem:one_over_norm_cp_continuous} the first term on the right hand side is bounded. This concludes the proof.
\end{proof}

Now we have all the tools together to prove the main theorem about the existence of minimizers.

\subsection*{Poof of Theorem~\ref{thm:ex_of_minimizers}}
Define the energy
\[
E(c) = \int_0^1 G_{c}(\dot c, \dot c ) \ud t\,,
\]
and the set
\[
\Om_{c_0, A}H^1 = \left\{ c \in H^1(I, \mc I^n) \,:\, c(0)=c_0,\, c(1) \in A \right\}\,,
\]
of curves starting at $c_0$ and ending in $A$. It is enough to show that $E$ attains a minimum on the set $\Om_{c_0, A}H^1$, since it is shown in \cite[Lem. 2.4.3]{Klingenberg1995}, that the minimum is a minimizing geodesic between $c_0$ and $c(1) \in A$. 

Let $(c^j)_{j \in \mathbb N}$ be a minimizing sequence. Then $E(c^j)$ is bounded and we let $r^2 > 0$ be an upper bound. We have the inequality
\[
\on{dist}(c_0, c^j(t)) \leq \sqrt{E(c^j)} \leq r \qquad \forall (t, j) \in I\x\mb N\,,
\]
and we see that all curves $c^j(t)$ lie in a metric ball around $c_0$ of radius $r$. This implies via Prop~\ref{prop:sobolev_uniform_equivalent} the existence of a constant $C>0$, s.t.
\begin{equation}
C\inv \| h \|_{H^n_\th} \leq \sqrt{G_{c^j(t)}(h,h)} \leq C \| h \|_{H^n_\th}
\end{equation}
holds for all $h \in H^n$ and all curves $c^j(t)$. 

As $E(c^j)$ is bounded for $j \in \mb N$ and $c^j(0) = c_0$, it follows from
\[
\| c^j(t) \|_{H^n_\th} \leq \| c_0 \|_{H^n_\th} + 
\| c^j(t) - c_0 \|_{H^n_\th} \leq \| c_0 \|_{H^n_\th} + 
C_1 \on{dist}(c_0, c^j(t)) \leq \| c_0 \|_{H^n_\th} + C_1R\,,
\]
with the constant $C_1$ given by Lem.~\ref{lem:dist_uniform_equivalence}, together with
\[
\| c^j \|_{H^1_t H^n_\th}^2 =
\int_0^1 \| c^j \|_{H^n_\th}^2 + \| \dot c^j \|_{H^n_\th}^2 \ud t
\leq \left( \| c_0 \|_{H^n_\th}^2 + C_1R \right)^2 + C^2 E(c^j)\,,
\]
that $\|c^j\|_{H^1_t H^n_\th}$ is bounded as well. Thus there exists a weakly convergent subsequence, again denoted by $(c^j)_{j \in \mb N}$, converging to $c^\ast \in H^1_t H^n_\th$. Let $\ep$ be chosen such that $n-\ep > 3/2$ and $0 < \ep < 1$. Since the embedding $H^1_tH^n_\th \hookrightarrow C_tH^{n-\ep}_\th$ is compact, by the Aubin--Dubinskii lemma---see, e.g., \cite{Amann2000}---we can further assume that $c^j \to c^\ast$ strongly in $C_tH^{n-\ep}_\th$.

From Prop.~\ref{prop:lipschitz_basic_results} we obtain another constant $C_2 = C_2(c_0, r)$, such that
\begin{equation}
\label{eq:lower_bound_mcp}
\left| \p_\th c^j(t,\th) \right| \geq C_2\qquad \forall \th \in S^1,\, \forall t \in I,\, \forall j \in \mathbb N\,,
\end{equation}
and because of the strong convergence in $C_t H_\th^{n-\ep}$ the bound remains valid for the limit as well. In particular this shows $c^\ast(t) \in \mc I^n$ for all $t \in I$.
Weak convergence in $H^1_t H^n_\th$ also shows $c^\ast(0) = c_0$ and $c^\ast(1) \in A$, since $A$ is weakly closed, and thus $c^\ast \in\Om_{c_0, A}H^1$.

It remains to show that $c^\ast$ is a minimizer for $E$. Because $G$ is a Sobolev metric with constant coefficients, we can write $E$ as
\[
E(c) = \sum_{k=0}^n a_k \left\| D_c^k \dot c \sqrt{|c'|} \right\|^2_{L^2_tL^2_\th}
\]
with constants $a_k > 0$. Here we write $D_c$ for $D_s$ to emphasize the dependence of the arc length derivative on the curve $c$. As $c^j$ is bounded in $H^1_t H^n_\th$ and $c^j \rightharpoonup c^\ast$ weakly in $H^1_t H^n_\th$, it follows from Lem.~\ref{lem:Ds_weak_convergence} that $D^k_{c^j} \dot c^j \rightharpoonup D^k_{c^\ast} \dot c^\ast$ weakly in $L^2_t L^2_\th$. Furthermore $c^j \to c^\ast$ in $C_t H^{n-\ep}_\th$ and hence $\sqrt{|\p_\th c^j|} \to \sqrt{|\p_\th c^\ast|}$ in $C_t H^{n-1-\ep}_\th$. Since $n-1-\ep > 1/2$, the pointwise product converges weakly,
\[
D^k_{c^j} \dot c^j \sqrt{|\p_\th c^j|} \rightharpoonup D^k_{c^\ast} \dot c^\ast \sqrt{|\p_\th c^\ast|} \text{ in } L^2_tL^2_\th\,,
\]
and since the norm-squared function $h \mapsto \| h\|^2$ is weakly sequentially lower semicontinuous, it follows that
\begin{multline*}
E(c^\ast) = \sum_{k=0}^n a_k \left\| D^k_{c^\ast} \dot c^\ast \sqrt{|\p_\th c^\ast|} \right\|^2_{L^2_2 L^2_\th} \leq \\ \leq \liminf_{j \to \infty}
\sum_{k=0}^n a_k \left\| D^k_{c^j} \dot c^j \sqrt{|\p_\th c^j|} \right\|^2_{L^2_2 L^2_\th}
\leq \liminf_{j \to \infty} E(c^j)\,.
\end{multline*}
Thus $c^\ast$ is a minimizer. \qed

\subsection{Space of Smooth Immersions}

We can also consider the question whether minimizing geodesics exist in the space $\on{Imm}(S^1,\R^d)$ of smooth curves. It is a characteristic property of geodesic equations on function spaces to preserve the smoothness of initial conditions. Let $G$ be a Sobolev metric of order $n$ and $(c_0, u_0)$ an initial position and velocity, that lie in $H^m$ with $m > n$ or even in $C^\infty$. Then  the geodesic with the given initial conditions will also lie in $H^m$ or $C^\infty$ respectively. This behaviour is shared by the Euler equation \cite{Ebin1970}, the Camassa-Holm equation \cite{Kouranbaeva1999,GayBalmaz2009}, geodesic equations of general Sobolev metrics on the diffeomorphism group \cite{Smolentsev2006} as well as on the space of curves \cite{Michor2007}, immersions \cite{Bauer2011b} or Riemannian metrics \cite{Bauer2013c} to name but a few examples \cite{Bauer2014}.

It is then tempting to argue as follows: given two smooth curves $c_0, c_1$, there exists a minimizing geodesic $c(t) \in \mc I^n(S^1,\R^d)$ connecting them. The geodesic cannot lose or gain smoothness and since the endpoints are smooth, so is the whole geodesic. Unfortunately this argument is flawed. To use the preservation of smoothness along the geodesic, we need to know about the smoothness of both the initial position $c_0$ and the initial velocity $u_0$. The map
\[
(c_0, u_0) \mapsto \left(\on{Exp}_{c_0} (u_0), \p_t|_{t=1} \on{Exp}_{c_0}(tu_0)\right)
\]
preserves smoothness. Whether the map
\[
(c_0, c_1) \mapsto \on{Log}_{c_0}(c_1)
\]
does the same is a different---and a more difficult---question.

A positive answer is given in \cite{Kappeler2008} for right-invariant Sobolev metrics in a neighborhood around the identity on the diffeomorphism group of the torus and one suspects that the proof can be generalized without too much difficulty to arbitrary compact manifolds. On the space of curves the problem remains open.

\begin{openquestion*}
Let $G$ be a Sobolev metric with constant coefficients of order $n\geq 2$. If $c(t)$ is a minimizing geodesic in $\mc I^n(S^1,\R^d)$ between $c_0$ and $c_1$ as given by Thm.~\ref{thm:ex_of_minimizers} and $c_0, c_1 \in \on{Imm}(S^1,\R^d)$, does it follow that $c(t) \in \on{Imm}(S^1,\R^d)$ for all $t \in I$? In other words, can any two curves in the same connected component of $\on{Imm}(S^1,\R^d)$ be joined by a minimizing geodesic?
\end{openquestion*}

\section{Shape Space}
\label{sec:shape_space}

\subsection{Quotient Spaces}

In this section we want to transfer the completeness results from $\mc I^n(S^1,\R^d)$ and $\on{Imm}(S^1,\R^d)$ to the shape space of unparametrized curves,
\[
B_i(S^1,\R^d) = \on{Imm}(S^1,\R^d)/\on{Diff}(S^1)\,.
\]
This space is a manifold, if we restrict ourselves to the regular orbits of the $\on{Diff}(S^1)$-action. Denote by $\on{Imm}_f(S^1,\R^d)$ the set of immersions upon which $\on{Diff}(S^1)$ acts freely. We have
\[
c \in \on{Imm}_f(S^1,\R^d) \;\text{ iff }\; \big(c \o \ph = c \;\Rightarrow\; \ph = \on{Id}_{S^1} \big)\,.
\]
The set $\on{Imm}_f$ is the open and dense set of regular points for the $\on{Diff}(S^1)$-action and we denote the quotient space by
\[
B_{i,f}(S^1,\R^d) = \on{Imm}_f(S^1,\R^d)/\on{Diff}(S^1)\,.
\]
It is shown in \cite[Sect. 1.5]{Michor1991} that $B_{i,f}$ is a smooth Fr\'echet manifold and the projection $\pi : \on{Imm}_f \to B_{i,f}$ is a smooth prinicpal fibration with structure group $\on{Diff}(S^1)$. The space $B_i$ is almost a manifold; for plane curves its singularities are described in \cite[Sect. 2.5]{Michor2006c}. Since $\on{Imm}_f$ is open and dense in $\on{Imm}$, so is $B_{i,f}$ in $B_i$.

We will also need the shape space of Sobolev immersions,
\[
\mc B^n(S^1,\R^d) = \mc I^n(S^1,\R^d)/ \mc D^n(S^1)\,.
\]
The space $\mc B^n$ does not appear to carry the structure of a manifold. To see this, note that for a plane curve $c \in \on{Imm}_f$, a chart arout $\pi(c) \in B_{i,f}$ is given by
\[
\Ph : \{ a\,:\,\| a \|_{C^1} < \ep \}\subset C^\infty \to B_{i,f}\,,\quad a \mapsto \pi(c + an_c)\,,
\]
with $\ep$ sufficiently small. However, if $c \in \mc I^n$, then the normal field $n_c$ lies only in $H^{n-1}$. Similarly the action of $\mc D^n(S^1)$ on $\mc I^n$ is only continuous and not smooth. We will show that the space $\mc B^n$ is the metric completion of $B_i$ and $B_{i,f}$.

While $\mc B^n$ may not be a manifold, it is a Hausdorff topological space. This can be shown more generally for the quotient of $\mc I^s(M,N)$, where $M$ is a compact and $N$ a finite-dimensional manifold, both without boundary. The following is a generalization of the results in \cite{Michor1991} to the Sobolev category.

\begin{proposition}
\label{prop:shape_space_hausdorff}
Let $M, N$ be finite-dimensional manifolds without boundary, $M$ compact and $s \in \R$ with $s > \on{dim}M/2+1$. Then $\mc D^s(M)$ acts continuously on $\mc I^s(M,N)$ and the quotient space $\mc I^s(M,N)/\mc D^s(M)$ is Hausdorff.
\end{proposition}

\begin{proof}
The continuity of the action is shown in \cite[Prop. 3.10]{Inci2013}. To show that the quotient space is Hausdorff, the proof of \cite[Thm. 2.1]{Michor1991}---showing the same statement for the quotient space $\on{Imm}_{\on{prop}}(M,N)/\on{Diff}(M)$ in the smooth category---can be reused without changes. Lemmas 2.2, 2.3, 2.5 and 2.8, Claim 2.6 as well as Construction 2.7 in \cite{Michor1991} are valid more generally for $C^1$-immersions.
\end{proof}

\subsection{Connectivity}
\label{sec:connectivity}

The connectivity of the space of immersions and of the shape space depends on the dimension $d$ of the ambient space. For $d=2$ the spaces $\on{Imm}_f$, $\on{Imm}$ and $\mc I^n$ decompose into connected components according to the degree of the curve \cite[Sect. 2.9]{Michor2006c}. The groups $\on{Diff}(S^1)$ and $\mc D^n(S^1)$ also have two connected components, the set of orientation-preserving and orientation-reversing diffeomorphisms. Orientation-preserving diffeomorphisms respect the degree of the curve while orientation-reversing diffeomorphisms map curves of degree $p$ to curves of degree $-p$. Denote by $\on{Imm}_p$ curves of degree $p$ and by $\on{Diff}^+(S^1)$ the orientation-preserving subgroup. Then the connected components $B_{i,p}$ of $B_i$ correspond to the non-negative degrees in the sense that $B_i = \bigcup_{p\geq 0} B_{i,p}$ and
\[
\pi\inv(B_{i,p}) = \on{Imm}_p \cup \on{Imm}_{-p}\,.
\]
For $p\neq 0$ we have
\[
B_{i,p} = \left( \on{Imm}_p \cup \on{Imm}_{-p} \right)/ \on{Diff}(S^1)
\cong \on{Imm}_p / \on{Diff}^+(S^1)\,,
\]
and the latter is a quotient of a connected space. For degree $p=0$ one simply has 
$B_{i,0} = \on{Imm}_0 / \on{Diff}(S^1)$. Similar statements hold for the spaces $\on{Imm}_f$ and $\mc I^n$. See \cite{Michor2006c, Kodama2006} for details.

For $d > 2$ the situation is simpler, since then $\on{Imm}_f$, $\on{Imm}$ and $\mc I^n$ are connected and path-connected and thus so are $B_{i,f}$, $B_i$ and $\mc B^n$.

\subsection{Completeness}

Now let us equip $\mc I^n$ with a Sobolev metric $G$ of order $n \geq 2$ with constant coefficients. Then $(\mc I^n, \on{dist})$ with the induced geodesic distance is a complete metric space and we can project the metric to a metric on $\mc B^n$ using the following general lemma.

\begin{lemma}
\label{lem:project_complete_metric}
Let $(X,d)$ be a metric space upon which the group $G$ acts by isometries. If the quotient space $X/G$ is Hausdorff, then
\[
d(G.x, G.y) := \inf_{g, h \in G} d(g.x, h.y) = \inf_{h \in G} d(x, h.y)
\]
defines a metric on $X/G$, that is compatible with the quotient topology on $X/G$. 

If $(X,d)$ is complete, then so is $(X/G,d)$.
\end{lemma}

\begin{proof}
Since $G$ acts on $X$ by isometries, we have $d(g.x, h.y) = d(x, g\inv h.y)$. Then
\[
d(G.x, G.z) = \inf_{g \in G} d(x, g.z) \leq d(x, h.y) + \inf_{g \in G} d(h.y, g.z)
= d(x, h.y) + d(G.y, G.z)\,.
\]
As $h \in G$ is arbitrary taking the infimum shows the triangle inequality. Symmetry is obvious, as is the property $d(G.x,G.x) = 0$.

To see that the topologies coincide, denote by $B_X(x,\ep)$ and $B_{X/G}(G.x,\ep)$ the open balls in $X$ and $X/G$ respectively and by $\pi : X \to X/G$ the canonical projection. Then
\begin{align*}
\pi\inv\left(B_{X/G}(G.x,\ep)\right) 
&= \left\{ y \,:\, \inf_{h\in G} d(x,h.y) < \ep \right\} \\
&= \left\{ g.y \,:\, g\in G,\, y \in B_X(x,\ep) \right\}
= G.B_X(x,\ep)
\end{align*}
and since $G.B_X(x,\ep)$ is open in $X$, it follows that $B_{X/G}(x,\ep)$ is open in $X/G$.

Now let $U \subseteq X/G$ be open, $G.x \in U$ and $\ep$ be such that $B_X(x,\ep) \subseteq \pi\inv(U)$. If $d(G.x,G.y) < \ep$ for some $y$, then $d(x, g.y) < \ep$ for some $g$ and hence $g.y \in B_X(x,\ep)$, implying $G.y = \pi(g.y) \in U$. Thus $B_{X/G}(G.x,\ep) \subseteq U$ and the topology induced by $d$ coincides with the quotient topology. As $X/G$ is assumed to be Hausdorff, it follows that $d(G.x,G.y)=0$ implies $G.x = G.y$.

Now let $(X,d)$ be complete and $(G.x_n)_{n \in \mb N}$ a Cauchy sequence. We can choose a subsequence, such that $d(G.x_n,G.x_{n+1}) < 2^{-n}$ holds for all $n\in \mb N$. Next we choose representatives of the orbit with $d(x_n,x_{n+1}) < d(G.x_n,G.x_{n+1}) + 2^{-n}$. Then
\[
d(x_n,x_{n+k}) \leq \sum_{i=n}^{n+k-1} d(x_i,x_{i+1}) \leq
\sum_{i=n}^{n+k-1} d(G.x_i,G.x_{i+1}) + 2^{-i}
\leq 2^{2-n}(1-2^{-k})\,,
\]
showing that $(x_n)_{n \in \mb N}$ is a Cauchy sequence in $X$. Let $x$ be the limit. Then
$\lim G.x_n = \lim \pi(x_n) = G.x$ and thus $(X/G,d)$ is complete.
\end{proof}

With the help of Lem.~\ref{lem:project_complete_metric} we can show that $(\mc B^n, \on{dist})$ is a complete metric space and furthermore the infimum in the definition of the quotient metric is attained.

\begin{theorem}
Let $n \geq 2$ and $G$ be a Sobolev metric of order $n$ with constant coefficients. Then $(\mc B^n(S^1,\R^d), \on{dist})$ with the quotient metric induced by the geodesic distance on $(\mc I^n(S^1,\R^d), G)$ is a complete metric space.

Given $C_1, C_2 \in \mc B^n(S^1,\R^d)$ in the same connected component, there exist $c_1, c_2 \in \mc I^n(S^1,\R^d)$ with $c_1 \in \pi\inv(C_1)$ and $c_2 \in \pi\inv(C_2)$, such that
\[
\on{dist}_{\mc B}(C_1, C_2) = \on{dist}_{\mc I}(c_1, c_2)\,;
\]
equivalently, the infimum in
\[
\on{dist}_{\mc B}(\pi(c_1), \pi(c_2)) = \inf_{\ph \in \mc D^n(S^1)} \on{dist}_{\mc I}(c_1, c_2 \o \ph)
\]
is attained.
\end{theorem}

\begin{proof}
It is shown in Prop.~\ref{prop:shape_space_hausdorff} that $\mc B^n(S^1,\R^d)$ is Hausdorff and in Thm.~\ref{thm:metric_completeness}, that $(\mc I^n(S^1,\R^d), \on{dist}_{\mc I})$ is complete. Then by Lem.~\ref{lem:project_complete_metric} it follows that $(\mc B^n(S^1,\R^d), \on{dist}_{\mc B})$ is a complete metric space.

Fix $c_1, c_2 \in \mc I^n(S^1,\R^d)$. As in the proof of Thm.~\ref{thm:ex_of_minimizers} we consider the energy
\[
E(c) = \int_0^1 G_{c}(\dot c, \dot c ) \ud t\,,
\]
this time on the set
\[
\Om_{c_1, c_2 \o \mc D^n(S^1)}H^1 = \left\{ c \in H^1(I, \mc I^n) \,:\, c(0)=c_1,\, c(1) \in c_2 \o \mc D^n(S^1) \right\}\,.
\]
We cannot use Thm.~\ref{thm:ex_of_minimizers} directly, because the orbit $c_2 \o \mc D^n(S^1)$ is closed, but we don't know if it also weakly closed. Instead, we choose a minimizing sequence $(c^j)_{j \in \mathbb N}$, i.e.,
\[
\lim_{j \to \infty} E(c^j) = \on{dist}_{\mc B}(\pi(c_1), \pi(c_2))\,.
\]
Following the proof of Thm.~\ref{thm:ex_of_minimizers} we can pass to a subsequence, again denoted by $(c^j)_{j \in \mathbb N}$, converging weakly $c^j \rightharpoonup c^\ast$ in $H^1_t H^n_\th$ and strongly $c^j \to c^\ast$ in $C_t H^{n-\ep}_\th$ to some $c^\ast \in H^1_t H^n_\th$, where $\ep$ is chosen such that $n-\ep > \frac 32$ and $0 < \ep < 1$. The property $E(c^\ast) \leq \liminf_{j \to \infty} E(c^j)$ continues to hold and we only need to show that $c^\ast(1) \in c_2 \o \mc D^n(S^1)$ to complete the proof.

To do this, note that $c^j(1) \in c_2 \o \mc D^n(S^1)$ for all $j \in \mathbb N$ and so using Prop.~\ref{prop:shape_space_hausdorff} and the strong convergence $c^j(1) \to c^\ast(1)$ in $H^{n-\ep}_\th$, we obtain $c^\ast(1) \in c_2 \o \mc D^{n-\ep}(S^1)$. Let $c^\ast(1) = c_2 \o \ph$ with $\ph \in \mc D^{n-\ep}(S^1)$. W.l.o.g. we assume that $c_2$ has constant speed and we differentiate,
\[
\left| c^\ast(1)' \right| = |c_2'| \o \ph\cdot \ph' = \frac{\ell_{c_2}}{2\pi} \ph'\,.
\]
This shows $\ph' \in H^{n-1}_{\th}$ and hence $\ph \in \mc D^n(S^1)$. Thus $c^\ast(1) \in c_2 \o \mc D^n(S^1)$ and we are done.
\end{proof}

\subsection{Length Space} The Sobolev shape space $(\mc B^n(S^1,\R^d), \on{dist})$ is also a length space.

\begin{theorem}
\label{thm:sob_shape_complete}
Let $n \geq 2$ and $G$ be a Sobolev metric of order $n$ with constant coefficients. Then $(\mc B^n(S^1,\R^d), \on{dist}_{\mc B})$ with the induced metric is a length space and any two shapes in the same connected component can be joined by a minimizing geodesic.
\end{theorem}

Here a minimizing geodesic is to be understood in the sense of metric spaces, i.e., a curve $\ga : I \to X$ into a metric space $(X,d)$ is a minimizing geodesic, if
\[
d(\ga(t), \ga(s)) = \la |t -s|
\]
holds for some $\la > 0$ and all $t, s \in I$; see \cite{Burago2001}.

\begin{proof}
Since $(\mc B^n(S^1,\R^d), \on{dist}_{\mc B})$ is a complete metric space, using \cite[Thm.~2.4.16]{Burago2001}, it is enough to show that for every $C_0, C_1 \in \mc B^n(S^1,\R^d)$ in the same connected component there exists a midpoint, that is a point $D$ with
\[
\on{dist}_{\mc B}(C_0,D) = \on{dist}_{\mc B}(D,C_1) = \frac 12 \on{dist}_{\mc B}(C_0,C_1)\,.
\]
Using Thm.~\ref{thm:sob_shape_complete} we can lift $C_0,C_1$ to $c_1,c_2 \in \mc I^n(S^1,\R^d)$ lying in the same connected component, such that $C_i = \pi(c_i)$ and
\[
\on{dist}_{\mc B}(C_0,C_1) = \on{dist}_{\mc I}(c_0, c_1)\,.
\]
Furthermore by Thm.~\ref{thm:ex_of_minimizers} there exists a minimizing geodesic $c(t)$ connecting $c_0, c_1$.

We claim that $\pi\left(c\left(\tfrac 12\right)\right)$ is a midpoint between $C_0, C_1$. Set $C(t) = \pi(c(t))$. If
\[
\on{dist}_{\mc B}\left(C_0, C\left(\tfrac 12\right)\right) = 
\on{dist}_{\mc I}\left(c_0, c\left(\tfrac 12\right)\right)
\text{ and }
\on{dist}_{\mc B}\left(C\left(\tfrac 12\right), C_1\right) = 
\on{dist}_{\mc I}\left(c\left(\tfrac 12\right), c_1\right)\,,
\]
then we are done. So assume that at least one of
\[
\on{dist}_{\mc B}\left(C_0, C\left(\tfrac 12\right)\right) < 
\on{dist}_{\mc I}\left(c_0, c\left(\tfrac 12\right)\right)
\text{ or }
\on{dist}_{\mc B}\left(C\left(\tfrac 12\right), C_1\right) < 
\on{dist}_{\mc I}\left(c\left(\tfrac 12\right), c_1\right)\,,
\]
holds with a strict inequality. Then using the triangle inequality we calculate
\begin{align*}
\on{dist}_{\mc B}(C_0,C_1) &\leq
\on{dist}_{\mc B}\left(C_0, C\left(\tfrac 12\right)\right) +
\on{dist}_{\mc B}\left(C\left(\tfrac 12\right), C_1\right) \\
&< \on{dist}_{\mc I}\left(c_0, c\left(\tfrac 12\right)\right) +
\on{dist}_{\mc I}\left(c\left(\tfrac 12\right), c_1\right) 
= \on{dist}_{\mc I}(c_0,c_1)\,,
\end{align*}
and thus arrive at a contradiction. Hence $C\left(\tfrac 12\right)$ is a midpoint between $C_0$ and $C_1$ and $(\mc B^n(S^1,\R^d), \on{dist}_{\mc B})$ is a complete length space.
\end{proof}

\subsection{Smooth Shape Spaces}
The dense inclusions $\on{Imm}_f \subset \on{Imm} \subset \mc I^n$ imply that the inclusions
\begin{equation}
\label{eq:shape_dense}
B_{i,f}(S^1,\R^d) \subset B_i(S^1,\R^d) \subset \mc I^n(S^1,\R^d)/\on{Diff}(S^1)
\end{equation}
are also dense. As $\on{Diff}(S^1) \subset \mc D^n(S^1)$, there is a natural continuous projection
\begin{equation}
\label{eq:sobolev_shape_proj}
\mc I^n(S^1,\R^d)/\on{Diff}(S^1) \to \mc I^n(S^1,\R^d)/\mc D^n(S^1) = \mc B^n(S^1,\R^d)\,.
\end{equation}
While this map is not injective, we claim that the composition
\[
B_i(S^1,\R^d) = \on{Imm}(S^1,\R^d)/ \on{Diff}(S^1) \to \mc I^n(S^1,\R^d) / \mc D^n(S^1) = \mc B^n(S^1,\R^d)
\]
is injective. Indeed, let $c_1, c_2 \in \on{Imm}(S^1,\R^d)$ and $c_2 = c_1 \o \ph$ with $\ph \in \mc D^n(S^1)$. By reparametrizing we can assume that $|c_1'|= \tfrac{\ell_{c_1}}{2\pi}$ is constant. Then we obtain by differentiating, $|c_2'| = \tfrac{\ell_{c_1}}{2\pi} \ph'$, and thus $\ph \in \on{Diff}(S^1)$ showing that $\pi(c_1) = \pi(c_2)$ already in $B_i(S^1,\R^d)$. Since 
the inclusions in \eqref{eq:shape_dense} are dense and the projection \eqref{eq:sobolev_shape_proj} is surjective, it follows that $B_i(S^1,\R^d)$ as well as $B_{i,f}(S^1,\R^d)$ are dense in $\mc B^n(S^1,\R^d)$.

Let $G$ be a Sobolev metric of order $n\geq 2$ with constant coefficients. We have seen in Thm.~\ref{thm:immersion_weak_completion} that the induced geodesic distance on $\on{Imm}(S^1,\R^d)$ coincides with the restriction of the induced geodesic distance on $\mc I^n(S^1,\R^d)$. We claim that the quotient metric on $B_i(S^1,\R^d)$ also coincides with the restriction of the quotient metric on $\mc B^n(S^1,\R^d)$. To see that, let $C_1, C_2 \in B_i$ and $C_1 = \pi(c_1)$, $C_2 = \pi(c_2)$. Then
\begin{align*}
\on{dist}_{B_i}(C_1,C_2) 
&= \inf_{\ph \in \on{Diff}(S^1)} \on{dist}_{\on{Imm}}(c_1, c_2 \o \ph) \\
&= \inf_{\ph \in \on{Diff}(S^1)} \on{dist}_{\mc I}(c_1, c_2 \o \ph) \\
&= \inf_{\ph \in \mc D^n(S^1)} \on{dist}_{\mc I}(c_1, c_2 \o \ph) = \on{dist}_{\mc B}(C_1, C_2)\,.
\end{align*}
We are allowed to pass from $\on{Diff}(S^1)$ to $\mc D^n(S^1)$ in the infimum, because $\mc D^n(S^1)$ acts continuously on $\mc I^n(S^1,\R^d)$ and $\on{dist}_{\mc I}$ induces the manifold topology.

The Riemannian metric $G$ on $\on{Imm}_f(S^1,\R^d)$ induces a smooth Riemannian metric on $B_{i,f}(S^1,\R^d)$---and if one wants to consider Riemannian metrics on orbifolds, it also induces a Riemannian metric on $B_{i}(S^1,\R^d)$. The geodesic distance of this metric coincides with the quotient distance from Thm.~\ref{lem:project_complete_metric}; see \cite{Michor2007}. This leads to the following result.

\begin{theorem}
Let $n \geq 2$ and $G$ be a Sobolev metric of order $n$ with constant coefficients. Then the metric completion of $B_i(S^1,\R^d)$  with the induced geodesic distance is $\mc B^n(S^1,\R^d)$. The same holds for $B_{i,f}(S^1,\R^d)$.
\end{theorem}

\begin{proof}
It was shown in Thm.~\ref{thm:sob_shape_complete} that $(\mc B^n, \on{dist}_{\mc B})$ is a complete metric space and we argued above that the inclusion $B_{i,f} \subset \mc B^n$ is isometric and dense. Hence $\mc B^n$ is the metric completion of $B_{i,f}$ and also of $B_i$.
\end{proof}

\appendix

\section{Geodesic distance on weak submanifolds}

We often encounter the following situation: let $G$ be a Sobolev metric of order $n\geq 2$ with constant coefficients and $m > n$. We can consider the Riemannian manifold $(\mc I^m(S^1,\R^d), G)$ and denote by $\on{dist}_{\mc I^m}$ the induced geodesic distance or we can look at the larger manifold $(\mc I^n(S^1,\R^d), G)$ with the geodesic distance $\on{dist}_{\mc I^n}$ and then restrict it to $\mc I^m(S^1,\R^d)$. Denote this restricted distance by $\on{dist}_{\mc I^n}\!|_{\mc I^m}$. What is the relationship between $\on{dist}_{\mc I^m}$ and $\on{dist}_{\mc I^n}\!|_{\mc I^m}$? It turns out that because $\mc I^m(S^1,\R^d)$ is dense in $\mc I^n(S^1,\R^d)$, they are the same. This allows us to talk of \emph{the} geodesic distance of a given Sobolev metric, without having to constantly reference the underlying space.

This is a more general phenomenon, that is best phrased using the notion of a weak submanifold, introduced in \cite{Eliasson1971}.

\begin{definition}
Let $M$, $M_0$ be manifolds modelled on convenient vector spaces. We call $M$ a weak submanifold of $M_0$, if around any point $x_0$ in the closure of $M$ in $M_0$, there exists a neighborhood $U_0$ in $M_0$ together with a chart $\ph_0 : U_0 \to \ph_0(U_0) \subseteq E_0$ for $M_0$ and a convenient vector space $E$, with a continuous inclusion $E \subseteq E_0$, such that the restriction of $\ph_0$ to $U = M \cap E_0$ is a chart $\ph : U \to \ph(U) = \ph_0(U_0) \cap E$ for $M$.
\end{definition}

We can show that for dense, weak submanifolds the restriction of the ambient geodesic distance coincides with the intrinsic one.

\begin{proposition}
\label{prop:geod_distance_submanifold}
Let $M_0$ be a separable Hilbert manifold, containing $M$ as a weak Hilbert submanifold, and $M$ be dense in $M_0$. Let $g_0$ be a weakly continuous Riemannian metric on $M_0$, i.e., the map $g_0 : TM_0 \x_{M_0} TM_0 \to \R$ is continuous, denote by $g$ its restriction to $M$ and by $d_0$ and $d$ the induced geodesic distances on $M_0$ and $M$ respectively. Then
\[
d_0|_M = d\,,
\]
i.e., the geodesic distance on $M_0$ restricted to $M$ coincides with the geodesic distance on $M$.
\end{proposition}

\begin{proof}
Let $x, y \in M$. Since $d(x,y)$ is defined by taking the infimum over paths in $M$ and $d_0$ by taking the infimum over paths in $M_0$, we have the inequality
\[
d_0(x,y) \leq d(x,y)\,.
\]
To show the other inequality, denote by $E$ and $E_0$ the separable Hilbert spaces upon which $M$ and $M_0$ are modelled. Then $E \subseteq E_0$ is dense and by Lem.~\ref{lem:approx_dense_hilbert} we can choose a family $P_n$ of linear operators $P_n : E_0 \to E$ with the property that
\begin{equation}
\label{eq:approx_hypothesis}
\lim_{n \to \infty} \| P_n v - v \|_{E_0} = 0
\end{equation}
and the convergence is uniform on compact subsets.

Let $\ga$ be a piecewise smooth curve in $M_0$ connecting $x$ and $y$. Assume w.l.o.g. that $x$ and $y$ can both be covered by a weak chart for $M$; otherwise we split the curve into a finite number of segments and apply the argument to each one. 

So we can assume that $M_0 \subseteq E_0$ is an open subset and $M = M_0 \cap E$. Then
\[
L_g(P_n \ga) = L_{g_0}(P_n \ga) = \int_0^1 \left| P_n \dot \ga\right|_{P_n \ga} \ud t\,.
\]
Since $g_0$ is weakly continuous, the integrand converges uniformly to $|\dot \ga|_{\ga}$ and thus
\[
L_g(P_n \ga) \to L_{g_0}(\ga) \text{ for } n \to \infty\,.
\]

We want to note the following: if $x_n \to x$ in $M_0$ and both $x, x_n \in M$, then also $d(x_n, x) \to 0$. This is related to the fact that the topology induced by a Riemannian metric is weaker than the manifold topology and thus convergence in the manifold topology implies convergence in the metric; however in this case we start with convergence in the topology of the ambient manifold $M_0$, but want to obtain a statement about the metric $d$ on $M$. To show this we use linear interpolation in a weak chart around $x$,
\[
d(x_n, x) \leq \int_0^1 |x_n - x|_{tx_n + (1-t)x} \ud t \leq \int_0^1 |x_n - x|_x \ud t + \ep\,,
\]
and $\ep$ comes from the continuity of $g_0$ on $M_0$.

Starting from
\[
d(x,y) \leq d(x, P_n \ga(0)) + L_{g}(P_n \ga) + d(P_n \ga(1), \ga)\,,
\]
the convergence $P_n \ga(0) \to x$ in $M_0$ implies $d(x, P_n \ga(0)) \to 0$ and hence in the limit
\[
d(x,y) \leq L_{g_0}(\ga)\,.
\]
As $\ga$ was arbitrary, this implies $d(x,y) \leq d_0(x,y)$ as required.
\end{proof}

\begin{remark} 
\label{rem:construct_P_by_hand}
If in Prop.~\ref{prop:geod_distance_submanifold} the manifold $M$ is not a Hilbert manifold, but modelled only on a convenient vector space, then the statement still holds, provided we have a family of linear operators $P_n: E_0 \to E$ with the property \eqref{eq:approx_hypothesis}, i.e.,
\begin{equation*}
\lim_{n \to \infty} \| P_n v - v \|_{E_0} = 0\,,
\end{equation*}
and uniform convergence on compact subsets. For Hilbert manifolds such an approximating family always exists, as shown below in Lem.~\ref{lem:approx_dense_hilbert}. If we want to relax the assumptions, the family has to be constructed by hand.
\end{remark}

\begin{lemma}
\label{lem:approx_dense_hilbert}
Let $E, E_0$ be two separable Hilbert spaces, $E$ continuously and densly embedded in $E_0$. Then there exists a sequence of bounded operators $P_n : E_0 \to E$, such that
\[
\forall x \in E_0:\, P_n x \to x \text{ in } E_0\,;
\]
in other words, $P_n \in L(E_0, E_0)$ converges in the strong operator topology to $\on{Id}_{E_0}$ and $\on{im} P_n \subseteq E$. 

Furthermore the convergence is uniform on compact subsets of $E_0$.
\end{lemma}

\begin{proof}
Given the two separable Hilbert spaces $E$ and $E_0$, the former continuously and densly embedded into the latter, \cite[Thm. 2.9]{Huet1976} shows the existance of an unbounded, self-adjoint operator $A : D(A) \to E_0$ with $D(A) = E$, representing the inner product,
\[
\langle v, w \rangle_{E} = \langle Av, Aw \rangle_{E_0} \text{ for } v, w \in E\,.
\]
Let $\{ P_\Om \}$ be the projection-valued measure associated to $A$. Then we have
\[
\lim_{n\to \infty} \|P_{[-n,n]} v - v\|_{E_0} = 0\,,
\]
and $P_{[-n,n]} v \in D(A) = E$. See, e.g., \cite[Sect. VIII.3]{Reed1980} for details.

To see that the convergence is uniform on compact subsets one uses that the operator norm of $P_{[-n,n]}$ is uniformly bounded.
\end{proof}

\providecommand{\href}[2]{#2}
\providecommand{\arxiv}[1]{\href{http://arxiv.org/abs/#1}{arXiv:#1}}
\providecommand{\url}[1]{\texttt{#1}}
\providecommand{\urlprefix}{URL }

\medskip
Received xxxx 20xx; revised xxxx 20xx.
\medskip

\end{document}